\documentclass{amsart}

\usepackage{amssymb}
\usepackage{amsrefs}
\usepackage{overpic}
\usepackage{enumitem}   
\usepackage{graphicx} 
\usepackage[hidelinks]{hyperref}

\graphicspath{{Figures/}} 

\newtheorem{theorem}{Theorem}
\newtheorem{proposition}{Proposition}

\newtheorem{lemma}{Lemma}

\newtheorem{question}{Question}

\begin{document}\title[A note on Hilbert 16th Problem]{A note on Hilbert 16th Problem}
	
\author[Armengol Gasull and Paulo Santana]
{Armengol Gasull$^1$ and Paulo Santana$^2$}
	
\address{$^1$ Departament de Matem\`{a}tiques, Facultat de Ci\`{e}ncies, Universitat Aut\`{o}noma de Barcelona, 08193 Bellaterra, Barcelona, Spain ; and Centre de Recerca Matem\`{a}tica, Edifici Cc, Campus de Bellaterra, 08193 Cerdanyola del Vall\`{e}s (Barcelona), Spain}
\email{armengol.gasull@uab.cat}
	
\address{$^2$ IBILCE--UNESP, CEP 15054--000, S. J. Rio Preto, S\~ao Paulo, Brazil}
\email{paulo.santana@unesp.br}
	
	
	
\begin{abstract}
	Let $\mathcal{H}(n)$ be the maximum number of limit cycles that a planar polynomial vector field of degree $n$ can have. In this paper we prove that $\mathcal{H}(n)$ is realizable by structurally stable vector fields with only hyperbolic limit cycles and that it is a strictly increasing function whenever it is finite.
\end{abstract}
	
\maketitle

\section{Introduction and statement of the main results}

Consider the planar polynomial system of differential equations $X=(P,Q)$ given by
\begin{equation}\label{0}
	\dot x= P(x,y), \quad \dot y=Q(x,y),
\end{equation}
where the dot means the derivative in relation to the independent variable $t$ and $P$, $Q\colon\mathbb{R}^2\to\mathbb{R}$ are polynomials. To system \eqref{0} corresponds a polynomial vector field $X=P\frac{\partial}{\partial x}+Q\frac{\partial}{\partial y}$ in the phase plane of the variables $x$ and $y$. In this paper we make no distinction between system \eqref{0} and its respective vector field. The degree of $X$ is the maximum of the degrees of $P$ and $Q$. Given $n\in\mathbb{N}$, let $\mathcal{X}^n$ be the set of the planar polynomial systems \eqref{0} of degree $n$, endowed with the coefficients topology. Given $X\in\mathcal{X}^n$, let $\pi(X)\in\mathbb{Z}_{\geqslant0}\cup\{\infty\}$ be its number of \emph{limit cycles} (i.e. isolated periodic orbits). 

In his famous address to the International Congress of Mathematicians in Paris 1900, David Hilbert raised his famous list of problems for the $20$th century \cite{Browder}, with the second part of the $16$th problem being about the limit cycles of planar polynomial vector fields. Hilbert asks if there is a uniform upper bound for the number of limit cycles of polynomial vector fields of degree $n$. More precisely, given $n\in\mathbb{N}$ let $\mathcal{H}(n)\in\mathbb{Z}_{\geqslant0}\cup\{\infty\}$ be given by,
	\[\mathcal{H}(n)=\sup\{\pi(X)\colon X\in\mathcal{X}^n\}.\]
Under this notation the second part of Hilbert's $16$th problem consists in obtaining an upper bound for $\mathcal{H}(n)$ and it is yet an open problem. Even for the quadratic case, it is not known if $\mathcal{H}(2)<\infty$. However, advances has been made and lower bounds for $\mathcal{H}(n)$ have been found. For small values of $n$, the best lower bounds so far are $\mathcal{H}(2)\geqslant 4$ \cites{ChenWang1979,Son1980}, $\mathcal{H}(3)\geqslant 13$ \cite{LiLiuYang2009} and $\mathcal{H}(4)\geqslant 28$ \cite{ProTor2019}. In general, it is known that $\mathcal{H}(n)$ increases at least as fast as $O(n^2\ln n)$ \cite{ChrLlo1995,HanLi2012,Lijibin2003}. However, although the known lower bounds are given by  strictly increasing functions, this does not imply that $\mathcal{H}(n)$ itself is strictly increasing. In our first main result we prove this fact.

\begin{theorem}\label{Main1}
	Given $n\in\mathbb{N}$, it holds $\mathcal{H}(n+1)\geqslant\mathcal{H}(n)+1$.
\end{theorem}

In particular, it follows from Theorem~\ref{Main1} that if $\mathcal{H}(n_0)=\infty$ for some $n_0\in\mathbb{N}$, then $\mathcal{H}(n)=\infty$ for every $n\geqslant n_0$. 

The proof of Theorem~\ref{Main1} is essentially a consequence from the fact that given $X\in\mathcal{X}^n$, we can embed $X$ into $\mathcal{X}^{n+1}$ and bifurcate one more limit cycle, while the others persist. This persistence follows from our second main result. To state it properly  we will remind  the notion of \emph{structural stability} and comment on its particularities when it is restricted to the polynomial case.

Roughly speaking, a smooth vector field is structurally stable if small perturbations do not change the topological character of its orbits. The hallmark work on this area is due to Peixoto~\cite{Pei1962} and his characterization theorem, which states that a $C^1$-vector field on a closed (i.e. compact and without boundary) two dimensional manifold is structurally stable if, and only if, the following statements hold.
\begin{enumerate}[label=(\alph*)]
	\item It has at most a finite number of singularities, all hyperbolic.
	\item It has at most a finite number of periodic orbits, all hyperbolic.
	\item It does not have saddle connections.
\end{enumerate}
Moreover, the family of structurally stable vector fields is open and dense in the set of all $C^1$-vector fields. For the structural stability of \emph{polynomial} vector fields endowed with the coefficients topology there are two main characterizations, given by Sotomayor \cite{Soto1985} and Shafer \cite{Sha1987}. The former defines structural stability of $X\in\mathcal{X}^n$ as the structural stability of its Poincar\'e compactification. The latter does not make use of this embedding and thus deals with new objects, such as saddles at infi\-ni\-ty. Hence, they obtained different sets of necessary and sufficient conditions for structural stability. Yet, there are many similarities. Let $X\in\mathcal{X}^n$. In both cases for $X$ to be structurally stable, statements $(a)$ and $(c)$ above are necessary and also the following weak version of statement $(b)$.
\begin{enumerate}
	\item[$(b')$] It has at most a finite number of periodic orbits, none of even multiplicity.
\end{enumerate}
So far it is not known if non-hyperbolic limit cycles of odd multiplicity are possible for a structurally stable vector field in the polynomial world. More precisely, there is the following open question.

\begin{question}[\cite{Soto1985,Sha1987}]\label{Q1}
	If $X\in\mathcal{X}^n$ has a non-hyperbolic limit cycle of odd multiplicity, then is $X$ structurally unstable in $\mathcal{X}^n$?
\end{question} 

Question~\ref{Q1} was explicitly raised by Sotomayor \cite[Problem $1.1$]{Soto1985} and Shafer \cite[Question $3.4$]{Sha1987} and kept both of them from obtaining necessary \emph{and} sufficient conditions for structural stability in $\mathcal{X}^n$. For more details, we refer to \cite{Santana}. Another important similarity between both works is the fact that structural stability is a generic property. That is, if we let $\Sigma^n\subset\mathcal{X}^n$ be the family of the structurally stable elements, then $\Sigma^n$ is open and dense, independently of the two approaches. Therefore, from now on we denote by $\Sigma^n$ the set of structurally stable vector fields of degree $n$ under either one of these two definitions.

Let $\Sigma^n_h\subset\Sigma^n$ be the family of structurally stable vector fields such that all their limit cycles are hyperbolic. In our second main result we prove that $\mathcal{H}(n)$ is realizable by the elements of this family.

\begin{theorem}\label{Main2}	
	For $n\in\mathbb{N}$, the following statements hold.
	\begin{enumerate}[label=(\alph*)]
		\item If $\mathcal{H}(n)<\infty$, then there is $X\in\Sigma_h^n$ such that $\pi(X)=\mathcal{H}(n)$.
		\item If $\mathcal{H}(n)=\infty$, then for each $k\in\mathbb{N}$ there is $X_k\in\Sigma_h^n$ such that $\pi(X_k)\geqslant k$.
	\end{enumerate}
\end{theorem}

Finally, due to its relation with the possible case of $\mathcal{H}(n)=\infty$, we also include at the end of this note a proof for the following folklore result: \emph{a planar analytic vector field has an enumerable number of limit cycles.}

The paper is organized as follows. In Section~\ref{Sec2} we recall some properties of \emph{rotated vector fields} and prove how they can be used to transform non-hyperbolic limit cycles in hyperbolic ones. The main theorems are proved in Section~\ref{Sec3}. In Section~\ref{Sec4} we prove the folklore result and provide some further remarks.

\section{Rotated vector fields}\label{Sec2}

Given a planar polynomial vector field $X=(P,Q)$, let $X_\alpha=(P_\alpha,Q_\alpha)$ be the one-parameter family given by
\begin{equation}\label{rot}
	P_\alpha=P\cos\alpha-Q\sin\alpha, \quad Q_\alpha= Q\cos\alpha+P\sin\alpha,
\end{equation}
with $\alpha\in\mathbb{R}$. Observe that $X_0=X$ and that $X_\alpha$ defines a \emph{completed family of rotated vector fields}, see Duff \cite{Duff}. Throughout out this paper, $X_\alpha$ will always denote the family given by \eqref{rot}.

In his seminal work Duff \cite{Duff} studied the properties of $X_\alpha.$ In particular, he proved the following result that we simply state for family \eqref{rot}, but that holds for more general $1$-parametric families of $\mathcal{C}^1$ vector fields.

\begin{theorem}[\cite{Duff}]\label{T1}
	Let $X_\alpha$ be the family of rotated vector fields \eqref{rot} and suppose that $X_{\alpha_0}$ has a limit cycle $\gamma_{\alpha_0}$. Then:
	\begin{enumerate}[label=(\alph*)]
		\item If $\gamma_{\alpha_0}$ has odd multiplicity, then it is persistent for $|\alpha-\alpha_0|$ small and it either contracts or expands monotonically as $\alpha$ varies in a certain sense. 
		\item If $\gamma_{\alpha_0}$ has even multiplicity, then for $|\alpha-\alpha_0|$ small it splits in two limit cycles, one stable and the other unstable, as $\alpha$ varies in a certain sense. If $\alpha$ varies in the opposite sense, then $\gamma_{\alpha_0}$ disappears and no other limit cycles appear in its neighborhood.
	\end{enumerate}
\end{theorem} 
We observe that Theorem~\ref{T1} does not provide information about the hyperbolicity of the limit cycles involved. However, it follows from Andronov et al \cite[Theorems~$71$\&$72$]{Andronov} that this information can be given in the analytic case. For sake of simplicity and for the paper to be self-contained, we provide a proof of a simple version of such theorems, sufficient for our goals.

\begin{proposition}\label{P1} Let $X_\alpha$ be the family of rotated vector fields \eqref{rot} and suppose that $X_{\alpha_0}$ has a limit cycle $\gamma_{\alpha_0}$. Then, for $|\alpha-\alpha_0|>0$ small enough, all the limit cycles detailed in Theorem~\ref{T1}   that bifurcate from $\gamma_{\alpha_0}$ are hyperbolic.
\end{proposition}

\begin{proof} For simplicity, let us assume $\alpha_0=0$. If $\gamma_0$ is hyperbolic, then there is nothing to prove. Hence, suppose that $\gamma_0$ is not hyperbolic. Let $I\subset\mathbb{R}$ be a small neighborhood of $\alpha_0=0$ and $\Sigma$ be a small normal section of $\gamma_{0}$, endowed with a coordinate system $s\in\mathbb{R}$ such that $s=0$ at $p$, where $\{p\}=\gamma_{0}\cap\Sigma$. Let $D\colon I\times\Sigma\to\mathbb{R}$ be its associated displacement map. Since $X_\alpha$ is analytic in $(x,y;\alpha)$, it follows that~$D$ is well defined and analytic. Let $T>0$ be the period of $\gamma_{0}$ and let $\gamma_{0}(t)$ be the parametrization of $\gamma_{0}$ given by the flow of $X_{0}$ and such that $\gamma_{0}(0)=p$. It follows from Perko \cite[Lemma $2$]{Perko1992} that, for some  $C\in\mathbb{R}\backslash\{0\},$
\begin{align}\label{2}
	\frac{\partial D}{\partial\alpha}(0,0)&=C\int_{0}^{T}\left(e^{-\int_{0}^{t}div(\gamma_{0}(\tau))\;d\tau}\right)X_\alpha\land\dfrac{\partial X_{\alpha}}{\partial\alpha}(\gamma_{0}(t);0)\;dt\\\nonumber
	&=C\int_{0}^{T}\left(e^{-\int_{0}^{t}div(\gamma_{0}(\tau))\;d\tau}\right)\big(P^2+Q^2)(\gamma_{0}(t);0)\;dt\ne0.
\end{align}	
Therefore, from the Implicit Function Theorem we have that there is a unique function $\alpha=\alpha(s)$, with $\alpha(0)=0$, such that 
\begin{equation}\label{4}
	D(\alpha(s),s)=0.
\end{equation}
Moreover, since $D$ is analytic, it follows that $\alpha(s)$ is also analytic. Differentiating~\eqref{4} in relation to $s$ we obtain,
\begin{equation}\label{5}
	\frac{\partial D}{\partial\alpha}(\alpha(s),s)\alpha'(s)+\frac{\partial D}{\partial s}(\alpha(s),s)=0.
\end{equation}
From \eqref{2} we have that $\frac{\partial D}{\partial\alpha}(\alpha(s),s)\neq0$ for $|s|$ small. Hence, it follows from \eqref{5} that,
\begin{equation}\label{6}
	\alpha'(s)=-\frac{\partial D/\partial s}{\partial D/\partial\alpha}(\alpha(s),s).
\end{equation}
Since $\gamma_{0}$ is not hyperbolic, it follows that $\frac{\partial D}{\partial s}(0,0)=0$ and thus from \eqref{6} we have $\alpha'(0)=0$.  Since $\alpha'$ is an analytic function, either $0$ is an isolated zero of $\alpha'$ or $\alpha'(s)\equiv0$ (and in particular $\alpha(s)\equiv0$) in a neighborhood of $s=0$. Let us discard this second possibility. In this case, from \eqref{4},  $D(0,s)\equiv0$ for $|s|$ small and thus $\gamma_0$ belongs to a continuous band of periodic orbits, contradicting the definition of limit cycle. Therefore, it follows from \eqref{5} that
\begin{equation*}
	\frac{\partial D}{\partial s}(\alpha(s),s)=-\frac{\partial D}{\partial \alpha}(\alpha(s),s)\alpha'(s)\neq0,
\end{equation*}
for $|s|>0$ small. Hence, any limit cycle of $X_\alpha$ near $\gamma_{0}$ is hyperbolic, for $|\alpha|>0$ small, as we wanted to prove. \end{proof}

We observe that Perko \cite[Theorem $3$]{Perko1992} also provided a similar result about the hyperbo\-licity of the limit cycles considered at Theorem~\ref{T1}$(b)$. For more details about the theory of rotated vector fields and its generalizations, we refer to~Han \cite{Han1999}, Perko \cite[Section $4.6$]{Perko2001} and the references therein.

\section{Proof of the main theorems}\label{Sec3}

Given $X=(P,Q)\in\mathcal{X}^n$, let $\pi_h(X)$ be its number of hyperbolic limit cycles. Observe that in general we have $\pi_h(X)\leqslant\pi(X)$. 

In this paper we also work with the possibility of $\pi(X)=\infty$ for some $X\in\mathcal{X}^n$. We choose to do this because although Il’yashenko~\cite{Shenko2} and \'Ecalle~\cite{Ecalle} independently claimed to have proved that this is impossible, it seems that some of their results start to be under discussion. For instance, in the recent work \cite{Yeung} a possible gap was found in Il’yashenko's proof. Our results are not based on these finiteness results. 

\begin{proposition}\label{P2}
	Let $X\in\mathcal{X}^n$. Then the following statements hold.
	\begin{enumerate}[label=(\alph*)]
		\item If $\pi(X)<\infty$, then there is $Y\in\mathcal{X}^n$ such that $\pi_h(Y)\geqslant\pi(X)$.
		\item If $\pi(X)=\infty$, then for each $k\in\mathbb{N}$ there is $Y_k\in\mathcal{X}^n$ such that $\pi_h(Y_k)\geqslant k$.
	\end{enumerate}
\end{proposition}

\begin{proof}Let $X\in\mathcal{X}^n$ and $X_\alpha$ be its respective family of rotated vector fields, given by \eqref{rot}. Let also:
\begin{enumerate}[label=(\roman*)]
	\item $h\in\mathbb{Z}_{\geqslant0}\cup{\infty}$ be the number of hyperbolic limit cycles of $X$;
	\item $m\in\mathbb{Z}_{\geqslant0}\cup{\infty}$ be the number of non-hyperbolic limit cycles $X$ of odd multiplicity;
	\item $m^\pm\in\mathbb{Z}_{\geqslant0}\cup{\infty}$ be the number of non-hyperbolic limit cycles $\gamma$ of $X$ of even multiplicity and such that $\gamma$ bifurcates in two hyperbolic limit cycles for $\pm\alpha>0$ small.	
\end{enumerate}
Observe that $\pi(X)=h+m+m^++m^-$. Suppose first $\pi(X)<\infty$. Without loss of generality, suppose $m^+\geqslant m^-$. It follows from Proposition~\ref{P1} that $X_\alpha$ has at least $h+m+2m^+$ hyperbolic limit cycles for $\alpha>0$ small enough. Hence, if we take $Y=X_\alpha$, then $Y\in\mathcal{X}^n$ and
	\[\pi_h(Y)\geqslant h+n+2m^+\geqslant h+n+m^++m^-=\pi(X).\]
If $\pi(X)=\infty$, then $h$, $m$, $m^+$ or $m^-$ are equal to infinity. In any case we apply the same reasoning on an sequence of vector fields having an increasing number of limit cycles, obtaining  the final desired sequence of vector fields. \end{proof} 

\begin{proof}[Proof of Theorem~\ref{Main2}.]
	
Suppose first $\mathcal{H}(n)<\infty$ and let $Z\in\mathcal{X}^n$ be such that $\pi(Z)=\mathcal{H}(n)$. It follows from Proposition~\ref{P2} that there is $Y\in\mathcal{X}^n$ such that $\pi_h(Y)\geqslant\pi(Z)$. Hence, it follows from the definition of $\mathcal{H}(n)$ that,
	\[\pi(Y)=\pi_h(Y)=\pi(Z)=\mathcal{H}(n).\]
Hence, every limit cycle of $Y$ is hyperbolic and any vector field in $\mathcal{X}^n,$ close enough to $Y,$ has also exactly $\mathcal{H}(n)$ limit cycles, all of them hyperbolic. In particular, there is an arbitrarily small perturbation $X\in\Sigma^n_h$ of $Y$ such that $\pi(X)=\mathcal{H}(n)$. 

Suppose now $\mathcal{H}(n)=\infty$. Observe that there is a sequence $(Z_j)$, with $Z_j\in\mathcal{X}^n$, such that $\pi(Z_j)\to\infty$ and $\pi(Z_j)<\infty$ for every $j\in\mathbb{N}$, or there is $Z\in\mathcal{X}^n$ such that $\pi(Z)=\infty$. In either case it follows from statement $(a)$ or $(b)$ of Proposition~\ref{P2}, respectively, that for each $k\in\mathbb{N}$ there is $Y_k\in\mathcal{X}^n$ such that $\pi_h(Y_k)\geqslant k$. Therefore, for each $k\in\mathbb{N}$ we can take a small enough perturbation $W_k\in\Sigma^n$ of $Y_k$ such that $\pi_h(W_k)\geqslant k$. It follows from the definition of $\Sigma^n$ that $\pi(W_k)<\infty$. Moreover, some of these limit cycles may be non-hyperbolic and with odd multiplicity. Thus, it follows similarly to the proof of Proposition~\ref{P2}, from the structural stability of $W_k$ and from the fact that $\Sigma^n$ is open and dense in $\mathcal{X}^n$, that we can take a small enough rotation $X_k\in\Sigma^n$ of $W_k$ such that the following statements hold.
\begin{enumerate}[label=(\roman*)]
	\item The hyperbolic limit cycles persist.
	\item The non-hyperbolic limit cycles become hyperbolic.
	\item $X_k$ and $W_k$ are topologically equivalent.
\end{enumerate}
In particular, it follows from $(iii)$ that we do not have the bifurcation of new limit cycles and thus we conclude that $X_k\in\Sigma_h^n$ and $\pi(X_k)\geqslant k$. \end{proof}

We now prove a technical lemma that we will need to proof Theorem~\ref{Main1}.

\begin{lemma}\label{L1}
	Let $X\in\mathcal{X}^n$ and $B\subset\mathbb{R}^2$ a closed ball centered at the origin. Then there is an arbitrarily small perturbation $Y$ of $X$ having a regular point $p\in\mathbb{R}^2\backslash B$ such that $\ell\cap B=\emptyset$, where $\ell$ is the straight line  $p+sY(p)$, $s\in\mathbb{R}$.
\end{lemma}

\begin{proof} It follows from Shafer \cite[Theorem~$3.2$]{Sha1987} that we can take an arbitrarily small perturbation $Y\in\mathcal{X}^n$ of $X$ such that $Y$ has at most a finite number of singularities. Let $Y=(P,Q)$ and let $P_i$ and $Q_i$, $i\in\{0,\dots,n\}$, be homogeneous polynomials of degree~$i$ such that $P=P_0+\dots+P_n$ and $Q=Q_0+\dots+Q_n$. Replacing $Y$ by an arbitrarily small perturbation if necessary, we can also suppose $P_n(1,0)Q_n(1,0)\neq0$. Let $p=(x,0)$, $x>0$. Since $Y$ has at most a finite number of singularities, there is $x_0>0$ such that if $x>x_0$, then $p$ is a regular point of $Y$. Let $\ell^+$ and $\ell^-$ be the two straight lines tangents to $B$ and passing through $p$. Let $\theta=\theta(x)$ be the angle between $\ell^\pm$ and the $x$-axis and observe that,
\begin{equation*}
	\lim\limits_{x\to\infty}\theta(x)=0.
\end{equation*} 
Let also $\varphi=\varphi(x)$ be the angle between $\ell$ and the $x$-axis, which is given by
\begin{equation*}
	\varphi(x)=\arctan\frac{Q(x,0)}{P(x,0)},
\end{equation*}
see Figure~\ref{Fig1}. 
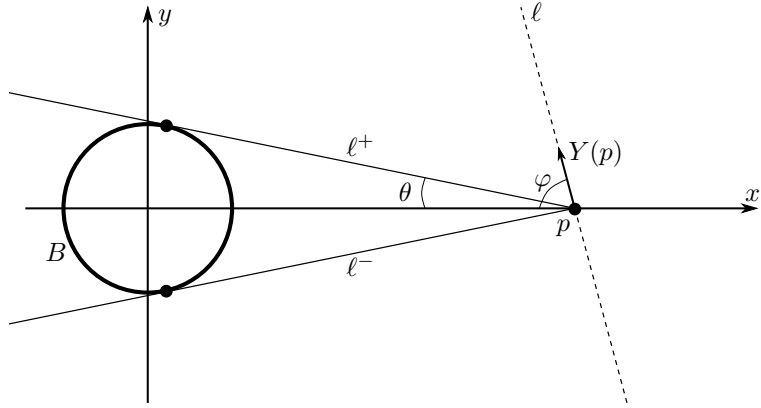
\begin{figure}[t]
	\begin{center}
		\begin{overpic}[width=10cm]{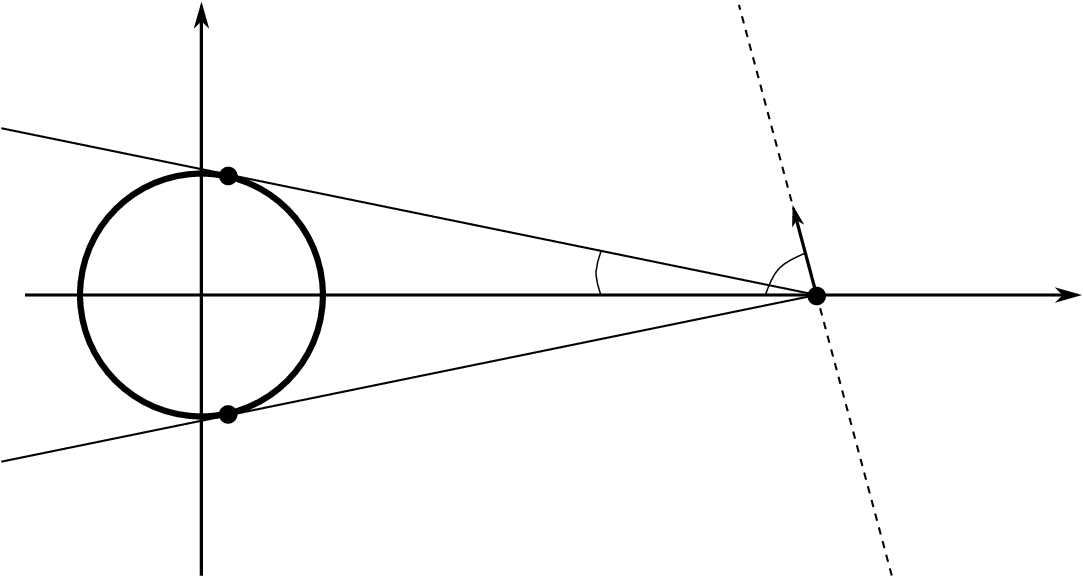} 
			\put(98,27){$x$}
			\put(20,51){$y$}
			\put(5,19){$B$}
			\put(45,33){$\ell^+$}
			\put(45,17){$\ell^-$}
			\put(73,23){$p$}
			\put(74.5,32.5){$Y(p)$}
			\put(52,27){$\theta$}
			\put(70,29){$\varphi$}
			\put(69.5,51){$\ell$}
		\end{overpic}
	\end{center}
	\caption{Illustration of $\ell^\pm$ and $\ell$.}\label{Fig1}
\end{figure}
Since $P_n(1,0)Q_n(1,0)\neq0$ 
it follows that,
	\[\lim\limits_{x\to\infty}\varphi(x)=\arctan\frac{Q_n(1,0)}{P_n(1,0)}\neq0.\]
As a consequence, $|\varphi(x)|>|\theta(x)|$ for $x>0$ big enough and thus $\ell\cap B=\emptyset$.
 \end{proof}

\begin{proof}[Proof of Theorem~\ref{Main1}] Suppose first $\mathcal{H}(n)<\infty$. It follows from Theorem~\ref{Main2}$(a)$ that there is $Z\in\Sigma_h^n$ such that $\pi(Z)=\mathcal{H}(n)$. Let $B\subset\mathbb{R}^2$ be a closed ball centered at the origin and such that all the limit cycles of $Z$ are in the interior of $B$. From Lemma~\ref{L1} and the structural stability of $Z$, we can suppose that $Z$ has a regular point $p\in\mathbb{R}^2\backslash B$ such that $p+sZ(p)\not\in B$ for every $s\in\mathbb{R}$. Let $Y=(P,Q)\in\Sigma_h^n$ be the vector field obtained from $Z$ by translating $p$ to the origin. Let $X=(R,S)\in\mathcal{X}^{n+1}$ be given by
	\[R(x,y)=(ax+by)P(x,y), \quad S(x,y)=(ax+by)Q(x,y),\]
with $a=-Q(0,0)$ and $b=P(0,0)$. Let $\ell\subset\mathbb{R}^2$ be the line given by $ax+by=0$ and observe that $X$ and $Y$ are equal on each connected component of $\mathbb{R}^2\backslash\ell$, except by the rescaling of time characterized by $dt/d\tau=ax+by$. It follows from Lemma~\ref{L1} that $B\cap\ell=\emptyset$ and thus $\pi_h(X)=\pi(X)=\pi(Y)$. Observe that $\ell$ is a line of singularities of $X$. In particular, the origin is a singularity of $X$ and its Jacobian matrix is given by,
	\[DX(0,0)=\left(\begin{array}{cc} aP(0,0) & bP(0,0) \vspace{0.2cm} \\ aQ(0,0) & bQ(0,0) \end{array}\right)=\left(\begin{array}{cc} ab & b^2 \vspace{0.2cm} \\ -a^2 & -ab \end{array}\right).\]
Hence, $\det DX(0,0)=0$ and $\text{Tr}\;DX(0,0)=0$. Let $X_{\varepsilon,\delta}=(R_\varepsilon,S_\delta)$ be given by
	\[R_\varepsilon(x,y)=(ax+(b+\varepsilon)y)P(x,y), \quad \quad S_\delta(x,y)=((a+\delta)x+by)Q(x,y),\]
and observe that we can take $|\varepsilon|>0$ and $|\delta|>0$ small enough such that the following statements hold.
\begin{enumerate}[label=(\roman*)]
	\item All the hyperbolic limit cycles inside $B$ persist.
	\item The origin is an isolated singularity.
	\item $\det DX_{\varepsilon,\delta}(0,0)>0$ and $\text{Tr}\;DX_{\varepsilon,\delta}(0,0)=0$.
\end{enumerate}
Hence, the origin is a monodromic singularity of $X_{\varepsilon,\delta}$. Let $L_1$ be its first \emph{Lyapunov constant} (see Adronov et al. \cite[p. 254]{Andronov}). Except perhaps by an arbitrarily small perturbation on the nonlinear terms of $X_{\varepsilon,\delta}$, we can suppose $L_1\neq0$. Therefore, we can take another small enough perturbation $W\in\mathcal{X}^{n+1}$ of $X_{\varepsilon,\delta}$ such that a limit cycle bifurcates from the origin, while the others persist. Hence we obtain
	\[\pi(W)\geqslant\pi(Y)+1=\mathcal{H}(n)+1,\]
and thus $\mathcal{H}(n+1)\geqslant\mathcal{H}(n)+1$. 

Suppose now $\mathcal{H}(n)=\infty$. It follows Theorem~\ref{Main2}$(b)$ that there is a sequence $(Z_k)$, with $Z_k\in\Sigma_h^n$, such that $\pi(Z_k)\to\infty$. Since $\pi(Z_k)<\infty$, we can apply the above reasoning on each $Z_k$ obtaining a sequence $(W_k)$, with $W_k\in\mathcal{X}^{n+1}$, such that $\pi(W_k)\to\infty$ and thus proving that $\mathcal{H}(n+1)=\infty$. \end{proof}

\section{Final remarks and a folklore result}\label{Sec4}

Theorem~\ref{Main1} is not the first known result about recurrence properties of $\mathcal{H}(n)$. It follows from the proof of  Christopher and Lloyd  \cite{ChrLlo1995} that $\mathcal{H}(2n+1)\geqslant4\mathcal{H}(n)$. Roughly speaking, given $X\in\mathcal{X}^n$, the authors translate all the limit cycles of $X$ to the first quadrant and thus apply the non-invertible transformation $(x,y)\mapsto(u^2,v^2)$, followed by the rescaling of time $dt/d\tau=2uv$. Hence, obtaining $Y\in\mathcal{X}^{2n+1}$ with a diffeomorphic copy of $X$ in each open quadrant.
	
The challenge of Theorem~\ref{Main1} has been to relate $\mathcal{H}(n+1)$ with $\mathcal{H}(n)$. It is much more easy for example to prove that $\mathcal{H}(n+2)\geqslant\mathcal{H}(n)+1$. Indeed, given $X\in\mathcal{X}^n$ let $Y=(x^2+y^2)X\in\mathcal{X}^{n+2}$ and observe that $Y$ is equivalent to $X$ except at the origin, where it has an extra degenerate singularity. Hence, similarly to the end of the proof of Theorem~\ref{Main1}, we can take a small perturbation of $Y$ creating an extra limit cycle.

We end this note with the following folklore result.

\begin{proposition}\label{Main3}
	Let $X$ be a planar analytic vector field. Then $X$ has an enumerable number of limit cycles. In particular, $\mathcal{H}(n)\leqslant\aleph_0$ for every $n\in\mathbb{N}$.
\end{proposition}

\begin{proof} If $X$ has no limit cycles, then there is nothing to prove. Suppose therefore that $X$ has at least one limit cycle and let $\Gamma=\{\gamma_a\}_{a\in A}$ be an indexation of all its limit cycles, $A\neq\emptyset$. For each $a\in A$, set 
	\[\delta_a=\inf\{d(\gamma_a,\gamma_b)\colon b\in A,\;b\neq a\},\]
where $d(\gamma_a,\gamma_b)$ is the usual distance between the compact sets $\gamma_a$ and $\gamma_b$,
	\[d(\gamma_a,\gamma_b)=\min\{||q_a-q_b||\colon q_a\in\gamma_a,\;q_b\in\gamma_b\}.\]
Since $X$ is analytic, it follows that $\gamma_a$ must be isolated (see \cite[p. $217$]{Perko2001}) and thus $\delta_a>0$ for every $a\in A$. Let $N_a\subset\mathbb{R}^2$ be the open $\delta_a/2$-neighborhood of $\gamma_a$, $a\in A$. Observe that if $a\neq b$, then $N_a\cap N_b=\emptyset$ (for otherwise $d(\gamma_a,\gamma_b)<\max\{\delta_a,\delta_b\}$). For each $a\in A$, choose $r_a\in N_a\cap\mathbb{Q}^2$ and define $i(a)=r_a.$ Observe that $r_a\neq r_b$ if $a\neq b$. Hence, we have an injective map $i\colon A\to\mathbb{Q}^2$ and thus $A$ is enumerable. \end{proof}

Notice that Proposition~\ref{Main3} is optimal for the analytic case. For instance, the planar analytic vector field 
	\[\dot x=-y+x\sin(x^2+y^2), \quad \dot y=x+y\sin(x^2+y^2),\]
has infinitely many limit cycles, given by $x^2+y^2=k\pi$, with $k\in\mathbb{Z}_{>0}$.

\section*{Acknowledgments}

This work is supported by the Spanish State Research Agency, through the projects PID2022-136613NB-I00 grant and the Severo Ochoa and Mar\'ia de Maeztu Program for Centers and Units of Excellence in R\&D (CEX2020-001084-M),  grant 2021-SGR-00113 from AGAUR, Generalitat de Ca\-ta\-lu\-nya, and by S\~ao Paulo Research Foundation (FAPESP), grants 2019/10269-3, 2021/01799-9 and 2022/14353-1.

\end{document}